\documentclass{amsart}
\usepackage{amsmath,amsthm,amssymb}
\usepackage{graphicx}

\newtheorem{theorem}{Theorem}[section]

\newtheorem{lemma}[theorem]{Lemma}
\newtheorem{corollary}[theorem]{Corollary}

\theoremstyle{definition}
\newtheorem{definition}[theorem]{Definition}
\newtheorem{problem}[theorem]{Problem}

\newtheorem{question}[theorem]{Question}

\theoremstyle{remark}
\newtheorem{remark}[theorem]{Remark}

\theoremstyle{theorem}

\title[Geometrically simply connected 4-manifolds]{Geometrically simply connected 4-manifolds and stable cohomotopy Seiberg-Witten invariants}
 
\author[Kouichi Yasui]{Kouichi Yasui}
\date{August 7, 2018. \textit{Revised}: May 4, 2019.}

\address{Department of Pure and Applied Mathematics, Graduate School of Information Science and Technology, Osaka University, 
1-5 Yamadaoka, Suita, Osaka 565-0871, Japan}
\email{kyasui@ist.osaka-u.ac.jp}
\begin{document}

\begin{abstract} We show that every positive definite closed 4-manifold with $b_2^+>1$ and without 1-handles has a vanishing stable cohomotopy Seiberg-Witten invariant, and thus admits no symplectic structure. We also show that every closed oriented 4-manifold with $b_2^+\not\equiv 1$ and $b_2^-\not\equiv 1\pmod{4}$ and without 1-handles admits no symplectic structure for at least one orientation of the manifold. In fact, relaxing the 1-handle condition, we prove these results under more general conditions which are much easier to verify. 
\end{abstract}

\maketitle

\section{Introduction}\label{sec:intro}
A compact connected 4-manifold is called \textit{geometrically simply connected}, if it admits a handle decomposition without 1-handles. The condition ``without 1-handles'' is equivalent to ``without 3-handles'' for a closed 4-manifold, as seen from dual decompositions. Clearly, every geometrically simply connected 4-manifold is simply connected, but the converse has been an open problem (\cite[Problem 4.18]{Kir78}). 
\begin{problem}\label{intro:problem:1-handle}
Is every simply connected closed smooth 4-manifold geometrically simply connected?
\end{problem}
This problem is closely related to the existence problem of \textit{exotic} (i.e.\ homeomorphic but not diffeomorphic) smooth structures on the two smallest 4-manifolds $S^4$ and $\mathbb{CP}^2$ (see \cite{Y08_AGT}), and many closed 4-manifolds were shown to be geometrically simply connected 
(e.g.\ \cite{H78}, \cite{M80}, \cite{AK80}, \cite{G91}, \cite{GS}). Furthermore, geometrically simply connected exotic smooth structures on the small 4-manifolds $\mathbb{CP}^2\#_n\overline{\mathbb{C}\mathbb{P}^2}$ $(6\leq n\leq 9$) were constructed by the author (\cite{Y08_AGT}, \cite{Y10}), and a long standing potential counterexample (\cite{HKK}) to Problem~\ref{intro:problem:1-handle} was disproved by Akbulut~\cite{A12} and independently by the author~\cite{Y08_JT}, but the problem remains unsolved. 

In this paper, we study gauge theoretical properties of geometrically simply connected closed 4-manifolds to reveal properties that hold for all simply connected closed 4-manifolds and/or to give potential methods for constructing counterexamples to Problem~\ref{intro:problem:1-handle}. Let us recall that a \textit{positive definite} 4-manifold is an oriented 4-manifold whose intersection form is positive definite. We first discuss the following question. 

\begin{question}\label{intro:question:definite} 
Does there exist a simply connected positive definite closed smooth 4-manifold that has a non-vanishing gauge theoretical invariant of smooth structures?
\end{question}
Any such 4-manifold would be an exotic $\#_n\mathbb{CP}^2$ for some $n$, but it has been an open problem whether $\#_n\mathbb{CP}^2$ admits an exotic smooth structure. Interestingly, Hom and Lidman~\cite{HL_JEMS} recently proved that, regarding the  Ozsv\'{a}th-Szab\'{o} 4-manifold invariant (with $\mathbb{Z}/2\mathbb{Z}$-coefficient) coming from Heegaard Floer homology, the answer to Question~\ref{intro:question:definite} is negative for geometrically simply connected 4-manifolds with $b_2^+>1$. 

\begin{remark}According to \cite{Zem} (see also \cite{JTZ}), the invariance of the Ozsv\'{a}th-Szab\'{o} 4-manifold invariant (\cite{OzSz_06}) is currently proved only for $\mathbb{Z}/2\mathbb{Z}$-coefficient.  This invariant is thus expected to be equivalent to the mod 2 version of the (ordinary) Seiberg-Witten invariant. We note that the Seiberg-Witten invariant is strictly stronger than its mod 2 version, that is, there exists an exotic pair of closed 4-manifolds that have distinct Seiberg-Witten invariants whose mod 2 versions are the same. Indeed, Fintushel-Stern's knot surgery~\cite{FS2} produces many such examples. 
\end{remark}

Here we answer Question~\ref{intro:question:definite} negatively for geometrically simply connected 4-manifolds with $b_2^+>1$, regarding the stable cohomotopy Seiberg-Witten invariant introduced by Bauer and Furuta~\cite{BF04}, which is strictly stronger than the Seiberg-Witten invariant (\cite{Bau04}). 

\begin{theorem}\label{intro:thm:definite:stable}Every geometrically simply connected positive definite closed smooth 4-manifold with $b_2^+>1$ has a vanishing stable cohomotopy Seiberg-Witten invariant. 
\end{theorem}

It is likely that our proof works for the $b_2^+=1$ case as well, but we do not pursue this point here, since this invariant requires some care in the $b_2^+=1$ case (see \cite{Bau04s}). We note that our approach is very different from that of Hom and Lidman. It would be natural to ask whether this theorem holds without the condition ``geometrically''.  If not, there exists a counterexample to Problem~\ref{intro:problem:1-handle}. 

We obtain the following corollary. 

\begin{corollary}\label{intro:cor:definite:SW}Every geometrically simply connected positive definite closed smooth 4-manifold with $b_2^+>1$ has a vanishing Seiberg-Witten invariant. 
\end{corollary}
\begin{proof}By \cite[Proposition~3.3]{BF04} and \cite[Proposition~4.4]{Bau04s}, a closed oriented smooth 4-manifold with a non-vanishing Seiberg-Witten invariant has a non-vanishing stable cohomotopy Seiberg-Witten invariant. Hence this corollary follows from the above theorem. 
\end{proof}

We note that Hom and Lidman proved their vanishing result on the Ozsv\'{a}th-Szab\'{o} 4-manifold invariant by utilizing the knot filtration on Heegaard Floer chain complex and its relationship with Dehn surgery, but their argument does not work for the Seiberg-Witten invariant due to lack of the corresponding tools in Seiberg-Witten theory. By contrast, we can give a short proof of Corollary~\ref{intro:cor:definite:SW} relying only on classical results about the (ordinary) Seiberg-Witten invariant.

The above corollary implies the following two results, which were originally proved by Hom and Lidman~\cite{HL_JEMS} using the vanishing result on the Ozsv\'{a}th-Szab\'{o} invariant. 

\begin{corollary}\label{intro:cor:definite:symplectic}Every geometrically simply connected positive definite closed smooth 4-manifold with $b_2^+>1$ admits no symplectic structure. 
\end{corollary}
\begin{proof}By a result of Taubes~\cite{Tau94}, the Seiberg-Witten invariant of a closed oriented symplectic 4-manifold with $b_2^+>1$ does not vanish (even for the mod 2 version). Hence the claim follows from Corollary~\ref{intro:cor:definite:SW}. 
\end{proof}
\begin{corollary}\label{intro:cor:definite:1- and 3-handles}If a simply connected positive definite closed symplectic 4-manifold admits a handle decomposition without 1- and 3-handles, then the 4-manifold is diffeomorphic to $\mathbb{CP}^2$. 
\end{corollary}
\begin{proof} Corollary~\ref{intro:cor:definite:symplectic} shows that a 4-manifold satisfying the assumption has $b_2^+=b_2=1$. The claim thus follows from the fact that a closed oriented smooth 4-manifold with $b_2=1$ having a handle decomposition without 1- and 3-handles is diffeomorphic to either $\mathbb{CP}^2$ or $\overline{\mathbb{C}\mathbb{P}^2}$ (see \cite[Proposition~6.4]{Y08_AGT}).
\end{proof}

We next discuss the following question. 
\begin{question}\label{intro:question:both ori}
Does there exist a simply connected closed oriented smooth 4-manifold with $b_2^+>1$ that admits symplectic structures for both orientations of the manifold?
\end{question}
We note that the answer to this question is affirmative, if either the condition ``simply connected'' or ``$b_2^+>1$'' is removed (e.g.\ $T^4$ and $S^2\times S^2$). 
A similar question for complex structures was intensively studied (\cite{Be85}, \cite{Ko92}, \cite{Ko97}), and several results of Kotschick (\cite{Ko92}, \cite{Ko97}) works for our question as well. For example, if a simply connected closed oriented 4-manifold with $b_2^+>1$ and $b_2^{-}>1$ admits symplectic structures for both orientations, then the 4-manifold does not contain a smoothly embedded 2-sphere representing a non-trivial second homology class, and both $b_2^+$ and $b_2^-$ are odd integers. Here we answer the question negatively for geometrically simply connected 4-manifolds with a mild condition on $b_2^+$ and $b_2^-$, giving a potential approach to Problem~\ref{intro:problem:1-handle}.

\begin{theorem}\label{intro:thm:reverse}Every geometrically simply connected closed oriented smooth 4-manifold with $b_2^+\not\equiv 1$ and $b_2^-\not\equiv 1\pmod{4}$ admits no symplectic structure for at least one orientation of the manifold. 
\end{theorem}


In fact, we prove our main results under more general conditions, relaxing the geometrically simply connected condition. These conditions are much easier to verify, and furthermore many closed 4-manifolds including non-simply connected ones satisfy these conditions. See Theorems~\ref{sec:proof:thm:definite}, \ref{sec:proof:thm:reverse} and Corollary~\ref{sec:proof:cor:definite:SW}. 

\section{Proof}\label{sec:proof}
We introduce the following definition to prove our main results. 

\begin{definition}
Let $X$ be an oriented smooth 4-manifold, and let $\alpha$ be a class of $H_2(X;\mathbb{Z})$. We say that $\alpha$ is represented by a \textit{2-handle neighborhood}, if $X$ has a codimension zero submanifold $W$ satisfying the following conditions. 
\begin{itemize}
 \item The submanifold $W$ is diffeomorphic to a 4-manifold obtained from the 4-ball by attaching a single 2-handle. (This submanifold will be called a 2-handle neighborhood.) 
 \item $\alpha$ is the image of a generator of $H_2(W;\mathbb{Z})\cong \mathbb{Z}$ by the inclusion induced homomorphism $H_2(W;\mathbb{Z})\to H_2(X;\mathbb{Z})$. 
\end{itemize}
\end{definition}

\begin{remark}According to \cite[Section 1]{LV02}, a second homology class $\alpha$ of a compact oriented smooth 4-manifold $X$ is represented by a 2-handle neighborhood, if and only if $\alpha$ is represented by a PL embedded 2-sphere in $X$. 
\end{remark}

For an oriented 4-manifold $X$, let $\overline{X}$ denote the 4-manifold $X$ equipped with the reverse orientation. We prove the following theorems. 

\begin{theorem}\label{sec:proof:thm:definite}If a closed connected positive definite smooth 4-manifold $X$ with $b_2^+>1$ admits a non-torsion second homology class represented by a 2-handle neighborhood, then the stable cohomotopy Seiberg-Witten invariant of $X$ vanishes. 
\end{theorem}

\begin{theorem}\label{sec:proof:thm:reverse}If a closed connected oriented smooth 4-manifold $X$ satisfying $b_2^+\not\equiv 1$, $b_2^-\not\equiv 1\pmod{4}$ and $b_1=0$ admits a non-torsion second homology class represented by a 2-handle neighborhood, then at least one of $X$ and $\overline{X}$ does not admit a symplectic structure. 
\end{theorem}

As we will see, Theorems~\ref{intro:thm:definite:stable} and \ref{intro:thm:reverse} easily follow from these theorems. We note that many closed 4-manifolds including non-simply connected ones admit non-torsion second homology classes represented by 2-handle neighborhoods (see \cite{GS}), and clearly this condition is  much easier to verify than the geometrically simply connected condition.  In fact, it is often not necessary to construct a handle decomposition of an entire 4-manifold. For example, there are many closed minimal symplectic 4-manifolds that contain cusp neighborhoods representing non-torsion classes and thus admit desired second homology classes (e.g.\ \cite{P07}). 

Theorem~\ref{sec:proof:thm:definite} implies the following corollary, as seen from the proofs of Corollaries~\ref{intro:cor:definite:SW} and \ref{intro:cor:definite:symplectic}. 

\begin{corollary}\label{sec:proof:cor:definite:SW}If a closed connected positive definite smooth 4-manifold $X$ with $b_2^+>1$ admits a non-torsion second homology class represented by a 2-handle neighborhood, then the Seiberg-Witten invariant of $X$ vanishes. Consequently, $X$ does not admit any symplectic structure. 
\end{corollary}

We begin the proofs of these theorems with the lemma below. For a second homology class $\alpha$ of an oriented 4-manifold $X$, let $\overline{\alpha}$ denote the class $\alpha$ of $\overline{X}$. 
\begin{lemma}\label{sec:proof:lem:slice}Let $X$ be a compact oriented smooth 4-manifold, and let $\alpha$ be a second homology class of $X$ represented by a 2-handle neighborhood. Then the class $\alpha-\overline{\alpha}$ of $H_2(X\#\overline{X};\mathbb{Z})\cong H_2(X;\mathbb{Z})\oplus H_2(\overline{X};\mathbb{Z})$ is represented by a smoothly embedded 2-sphere with the self-intersection number zero. 
\end{lemma}
\begin{proof}Assume that $\alpha$ is represented by a 2-handle neighborhood $W$ that is obtained from the 4-ball by attaching a 2-handle along an $n$-framed knot $K$. Then $X\#\overline{X}$ contains the boundary connected sum $W\natural \overline{W}$ as a submanifold. Let $\overline{K}$ denote the mirror image of the knot $K$. Clearly $W\natural \overline{W}$ is obtained from the 4-ball by attaching two 2-hanldles along an $n$-framed knot $K$ and a $(-n)$-framed knot $\overline{K}$, where these two framed knots are located in two disjoint 3-balls in $S^3$. By sliding the 2-handle $K$ over $\overline{K}$, we obtain a new 2-handle of $W\natural \overline{W}$ attached along the slice knot $K\# \overline{K}$ with the 0-framing. Clearly $\alpha-\overline{\alpha}$ is represented by this 2-handle neighborhood. Since $K\# \overline{K}$ is a slice knot, and the framing is zero, the class $\alpha-\overline{\alpha}$ is represented by a smoothly embedded 2-sphere with the self-intersection number zero. 
\end{proof}

Let us recall a few basic results about the stable cohomotopy Seiberg-Witten invariant of 4-manifolds~\cite{BF04}, also known as the Bauer-Furuta invariant. As shown in \cite{BF04} and \cite{Bau04s}, this invariant is a refinement of the Seiberg-Witten invariant, and moreover strictly stronger than the Seiberg-Witten invariant. Indeed, the following theorem of Bauer implies that this invariant can distinguish 4-manifolds having the same (vanishing) Seiberg-Witten invariants. 

\begin{theorem}[{Bauer~\cite{Bau04}, see also \cite[Theorem~8.8]{Bau04s}}]\label{sec:proof:thm:Bauer}If a closed connected oriented smooth 4-manifold $X$ satisfies either the condition $(1)$ or $(2)$, then the stable cohomotopy Seiberg-Witten invariant of $X$ does not vanish. 
\begin{enumerate}
 \item $X$ is the connected sum $X_1\# X_2$ of closed connected oriented smooth 4-manifolds $X_1$ with a non-vanishing stable cohomotopy Seiberg-Witten invariant and $X_2$ with $b_2^+(X_2)=0$. 
 \item $X$ is the connected sum $\#_{i=1}^n X_i$ of closed connected oriented smooth 4-manifolds $X_1, X_2,\dots, X_n$ satisfying the following conditions. 
 \begin{itemize}
 \item [\textnormal{(i)}] $b_2^+(X_i)\equiv 3 \pmod{4}$ and $b_1(X_i)=0$ for each $i$. 
 \item [\textnormal{(ii)}] Each $X_i$ admits a $spin^c$ structure $\mathfrak{s}_i$ compatible with an almost complex structure satisfying $SW_{X_i}(\mathfrak{s}_i)\equiv 1\pmod{2}$. 
 \item [\textnormal{(iii)}] $2\leq n \leq 4$. Furthermore, if $n=4$, then $b_2^+(X)\equiv 4\pmod{8}$. 
\end{itemize}
\end{enumerate}
\end{theorem}

Furthermore, Ishida and Sasahira~\cite{IS15} extended the sufficient condition (2) to the case $b_1\neq 0$. For interesting examples and applications of these results, the readers can consult, for example, \cite{IL03}, \cite{AIP14}, \cite{BI14}, and \cite{IS17}. 

As is well-known to experts of Seiberg-Witten theory, the adjunction inequality holds for the stable cohomotopy Seiberg-Witten invariant as well (e.g.\ \cite[p.\ 53]{Kr99}, \cite{Sas06}). In particular, the following special case holds. 

\begin{theorem}\label{sec:proof:thm:non-negative}Let $X$ be a closed connected oriented smooth 4-manifold with $b_2^+>1$ having a non-vanishing stable cohomotopy Seiberg-Witten invariant, and let $\alpha$ be a non-torsion second homology class of $X$. If the self-intersection number of $\alpha$ is non-negative, then $\alpha$ cannot be represented by a smoothly embedded 2-sphere. 
\end{theorem}

This theorem follows, for example, from the theorem below. 
\begin{theorem}[{Fr\o yshov~\cite[Theorem~1.1]{Fr05}}]\label{sec:proof:thm:Froyshov}
Let $X$ be a closed connected oriented smooth 4-manifold with $b_2^+>1$. Suppose that a closed orientable codimension one submanifold $Y$ of $X$ satisfies the following two conditions. 
\begin{itemize}
 \item $Y$ admits a Riemannian metric with positive scalar curvature. 
 \item The inclusion induced homomorphism $H^2(X;\mathbb{Q})\to H^2(Y;\mathbb{Q})$ is non-zero. 
\end{itemize}
Then the stable cohomotopy Seiberg-Witten invariant of $X$ vanishes. 
\end{theorem}

Although Theorem~\ref{sec:proof:thm:non-negative} follows from the above theorem by a standard argument, we include a proof for completeness. See also a recent preprint \cite{KLS} for an alternative proof that uses relative Bauer-Furuta invariants. 

\begin{proof}[Proof of Theorem~\ref{sec:proof:thm:non-negative}]Let $n\geq 0$ be the self-intersection number of $\alpha$, and let $Z$ be the 4-manifold $X\#_n\overline{\mathbb{C}\mathbb{P}^2}$. We note that $Z$ has a non-vanishing stable cohomotopy Seiberg-Witten invariant by Theorem~\ref{sec:proof:thm:Bauer}. 

Now suppose, to the contrary, that $\alpha$ is represented by a smoothly embedded 2-sphere in $X$. Then, by blowing up, one can construct a smoothly embedded 2-sphere $S$ in $Z$ with the self-intersection number zero that represents a non-torsion second homology class. 
Let $Y$ denote the boundary of the tubular neighborhood $\nu(S)(\cong S^2\times D^2)$ of $S$ in $Z$. We note that $Y$ is diffeomorphic to $S^2\times S^1$, and thus admits a Riemannian metric with positive scalar curvature. 
Since $S$ represents a non-torsion second homology class, we see that the inclusion induced homomorphism $H^2(Z;\mathbb{Q})\to H^2(\nu(S);\mathbb{Q})\cong \mathbb{Q}$ is non-zero. Composing this map with the inclusion induced homomorphism $H^2(\nu(S);\mathbb{Q})\to H^2(Y;\mathbb{Q})\cong \mathbb{Q}$, one can check that the inclusion induced homomorphism $H^2(Z;\mathbb{Q})\to H^2(Y;\mathbb{Q})\cong \mathbb{Q}$ is non-zero. Therefore Theorem~\ref{sec:proof:thm:Froyshov} shows that the stable cohomotopy Seiberg-Witten invariant of $Z$ vanishes, giving a contradiction. 
\end{proof}

We can now easily prove Theorems~\ref{sec:proof:thm:definite} and \ref{sec:proof:thm:reverse}. 

\begin{proof}[Proof of Theorem~\ref{sec:proof:thm:definite}]Let $X$ be a closed connected positive definite smooth 4-manifold with $b_2^+>1$, and assume that a non-torsion second homology class $\alpha$ of $X$ is represented by a 2-handle neighborhood. Suppose, to the contrary, that the stable cohomotopy Seiberg-Witten invariant of $X$ does not vanish. Since the intersection form of $\overline{X}$ is negative definite, the stable cohomotopy Seiberg-Witten invariant of $X\#\overline{X}$ does not vanish by Theorem~\ref{sec:proof:thm:Bauer}. Hence by Theorem~\ref{sec:proof:thm:non-negative}, $X\#\overline{X}$ does not contain a smoothly embedded 2-sphere with the self-intersection number zero representing a non-torsion second homology class.  On the other hand, by Lemma~\ref{sec:proof:lem:slice}, the non-torsion class $\alpha-\overline{\alpha}$ is represented by such a 2-sphere in $X\#\overline{X}$, giving a contradiction.
\end{proof}

\begin{proof}[Proof of Theorem~\ref{sec:proof:thm:reverse}]Let $X$ be a closed connected oriented smooth 4-manifold satisfying $b_2^+\not\equiv 1$, $b_2^-\not\equiv 1\pmod{4}$ and $b_1=0$, and assume that $X$ admits a non-torsion second homology class $\alpha$ represented by a 2-handle neighborhood. In the case where either $b_2^+$ or $b_2^-$ is an even integer, the claim immediately follows from the well-known fact that $b_2^+-b_1$ of a closed symplectic 4-manifold is odd (see \cite[Corollary~10.1.10]{GS}). We thus consider the case where $b_2^+\equiv b_2^-\equiv 3\pmod{4}$. We may assume that $X$ admits a symplectic structure. Suppose, to the contrary, that $\overline{X}$ also admits a symplectic structure. Then by a result of Taubes~\cite{Tau94}, both $X$ and $\overline{X}$ satisfies the condition (2)(ii) of Theorem~\ref{sec:proof:thm:Bauer}. Since $b_2^+(X)\equiv b_2^+(\overline{X})\equiv 3\pmod{4}$, Theorem~\ref{sec:proof:thm:Bauer} thus shows that the stable cohomotopy invariant of $X\#\overline{X}$ does not vanish. Hence by Theorem~\ref{sec:proof:thm:non-negative}, $X\#\overline{X}$ does not contain a smoothly embedded 2-sphere  with the self-intersection number zero representing a non-torsion second homology class. On the other hand, by Lemma~\ref{sec:proof:lem:slice}, the non-torsion class $\alpha-\overline{\alpha}$ of $X\#\overline{X}$ is represented by such a 2-sphere, giving a contradiction. 
\end{proof}

Theorems~\ref{intro:thm:definite:stable} and \ref{intro:thm:reverse} easily follow from Theorems~\ref{sec:proof:thm:definite} and \ref{sec:proof:thm:reverse}. 

\begin{proof}[Proof of Theorems~\ref{intro:thm:definite:stable} and \ref{intro:thm:reverse}]We note that, for any compact 4-dimensional handlebody with $b_2\neq 0$ and without 1-handles, the handlebody has a 2-handle representing a non-torsion second homology class, since the second homology group is generated by 2-handles. Theorem~\ref{intro:thm:definite:stable} thus follows from Theorem~\ref{sec:proof:thm:definite}. For Theorem~\ref{intro:thm:reverse}, we may assume $b_2\neq 0$, since any simply connected closed 4-manifold with $b_2=0$ does not admit a symplectic structure.  Theorem~\ref{intro:thm:reverse} thus follows from Theorem~\ref{sec:proof:thm:reverse}. 
\end{proof}

\begin{remark} We can prove Corollary~\ref{intro:cor:definite:SW} and, more generally, the $b_1(X)=0$ case of Corollary~\ref{sec:proof:cor:definite:SW} (and hence also Corollaries~\ref{intro:cor:definite:symplectic} and \ref{intro:cor:definite:1- and 3-handles}) without using the stable cohomotopy Seiberg-Witten invariant. Indeed, as seen from the proof of Theorem~\ref{intro:thm:definite:stable}, these corollaries can be shown by using the blow-up formula (\cite{FS95}, \cite[Proposition 2]{KMT95}, \cite[Corollary~14.1.1]{Fr08}) and the adjunction inequality (\cite{KM94}, \cite{MST}, \cite{FS95}, see also \cite[Theorem 2.4.8]{GS}) for the Seiberg-Witten invariant together with Lemma~\ref{sec:proof:lem:slice}. Note that the blow-up formula holds for a connected sum with an arbitrary closed negative definite 4-manifold satisfying $b_1=0$ (\cite[Proposition 2]{KMT95}). 
\end{remark}
\begin{remark}
(1) Problem 4.18 in the Kirby's problem list \cite{Kir78} asks not only Problem~\ref{intro:problem:1-handle} in this paper but also the stronger problem of whether every simply connected closed oriented smooth 4-manifold admits a handle decomposition without 1- and 3-handles. Indeed, many 4-manifolds were shown to admit such handle decompositions (e.g.\ references mentioned in Section~\ref{sec:intro}). Also, Rasmussen's paper \cite{Ra} states the vanishing of the Ozsv\'{a}th-Szab\'{o} 4-manifold invariants for homotopy $S^2\times S^2$'s without 1- and 3-handles. \medskip\\
(2) For simply connected closed 4-manifolds having handle decompositions without 1- and 3-handles, the proofs of Theorems~\ref{intro:thm:definite:stable}, \ref{intro:thm:reverse} and Corollary~\ref{intro:cor:definite:SW} can be simplified using the following fact: for any simply connected closed oriented smooth 4-manifold $X$ without 1- and 3-handles, $X\#\overline{X}$ is diffeomorphic to either $\#_n S^2\times S^2$ or $\#_n(\mathbb{CP}^2\#\overline{\mathbb{C}\mathbb{P}^2})$, where $n=b_2(X)$ (see \cite[Corollary 5.1.6]{GS}). 
\end{remark}

\section{Questions}
Finally we discuss two more questions, motivated by Problem~\ref{intro:problem:1-handle} and our results. We note the following lemma. 

\begin{lemma}\label{sec:question:lem:geometrically}If $X$ is a geometrically simply connected compact oriented smooth 4-manifold, then every second homology class of $X$ is represented by a 2-handle neighborhood. 
\end{lemma}
\begin{proof}We fix a handle decomposition of $X$ having no 1-handles, and consider the 2-chain group generated by 2-handles of the decomposition.  Let $\alpha$ be a second homology class  of $X$. Then $\alpha$  is represented by a linear combination of 2-handles. By introducing a cancelling pair of 2- and 3-handles, and sliding the newly introduced 2-handle over the original 2-hanldes, one can construct a 2-handle that is homologous to the linear combination, showing that $\alpha$ is represented by this 2-handle neighborhood. Note that the newly introduced 2-handle represents the zero element in the second homology group, and each handle slide corresponds to an addition or subtraction in the 2-chain group. 
\end{proof}

Now, it would be natural to ask the following questions. 

\begin{question}(1) Does every simply connected closed oriented smooth 4-manifold with $b_2\neq 0$ admit a non-zero second homology class represented by a 2-handle neighborhood?\medskip\\
(2) For any simply connected closed oriented smooth 4-manifold, is every second homology class represented by a 2-handle neighborhood?
\end{question}

If the answer to the question (1) is affirmative, then Theorems~\ref{intro:thm:definite:stable} and \ref{intro:thm:reverse} hold even in the case where the condition ``geometrically simply connected'' is replaced with ``simply connected'', as seen from the proofs. If the answer to the stronger question (2) is negative, then by Lemma~\ref{sec:question:lem:geometrically}, there exists a counterexample to Problem~\ref{intro:problem:1-handle}. We note that the answer to the question (1) is negative, if the simply connected condition is removed. Indeed, the product of two closed oriented surfaces of positive genera is a non-simply connected counterexample, since this 4-manifold satisfies $b_2\neq 0$ and $\pi_2=0$, but a second homology class given by a 2-handle neighborhood must be represented by an immersed 2-sphere. Of course, this argument does not work for simply connected 4-manifolds. 

In \cite{Y_stably}, we will answer the question (1) negatively for simply connected non-closed 4-manifolds, namely, we will show that there exists a simply connected compact oriented smooth 4-manifold that does not admit a non-zero second homology class represented by a 2-handle  neighborhood (and hence by a PL embedded 2-sphere). In fact, we will produce many such examples including those homotopy equivalent to $S^2$. Moreover, we will show that this property does depend on the choice of smooth structures of such a 4-manifold. We will prove these results by applying ideas of this paper to new type of exotic 4-manifolds constructed in \cite{Y_stably}. 

\subsection*{Acknowledgements}The author is grateful to Nobuhiro Nakamura and Hirofumi Sasahira for many helpful conversations and comments about the stable cohomotopy Seiberg-Witten invariant. The author would like to thank R.\  \.{I}nan\c{c} Baykur, Mikio Furuta and Hokuto Konno for useful comments. The author would also like to thank the referee for many helpful comments which improved the presentation of this paper. 
The author was partially supported by JSPS KAKENHI Grant Numbers 16K17593, 26287013 and 17K05220. 

\end{document}